\documentclass[a4paper, 10 pt, twocolumn]{IEEEtran}  % Comment this line out if you need a4paper

\usepackage{graphics} % for pdf, bitmapped graphics files
\usepackage{epsfig} % for postscript graphics files
\usepackage{mathptmx} % assumes new font selection scheme installed
\usepackage{times} % assumes new font selection scheme installed
\usepackage{amsmath} % assumes amsmath package installed
\usepackage{amssymb}  % assumes amsmath package installed

\usepackage{enumitem} %per alcuni stili enumerate

\usepackage{amsthm,mathtools}
\usepackage{xcolor} %per usare \textcolor
\usepackage{floatrow}
\usepackage{hyperref}

\theoremstyle{plain}
\newtheorem{theorem}{Theorem}[section]

\newtheorem{corollary}[theorem]{Corollary}

\theoremstyle{definition}

\theoremstyle{remark} 
\newtheorem{remark}[theorem]{Remark}

\newcommand{\eps}{\varepsilon}
\newcommand{\lam}{\lambda}

\newcommand{\norm}[1]{\left\lVert {#1} \right\rVert}

\title{\LARGE \bf
	Convergence of the Value Function in Optimal Control Problems with Unknown Dynamics}
\date{}

\author{Andrea Pesare$^{1}$, Michele Palladino$^{2}$ and Maurizio Falcone$^{1}$% <-this % stops a space
\thanks{ $^{1}$A. Pesare and M. Falcone are with the Deparment of Mathematics, Sapienza Universit\`a di Roma, 00185 Rome, Italy
       ({\tt\small e-mail: pesare@mat.uniroma1.it, falcone@mat.uniroma1.it})}%
\thanks{$^{2}$M. Palladino is with the Gran Sasso Science Institute - GSSI, 67100 L'Aquila, Italy
        ({\tt\small e-mail: michele.palladino@gssi.it})}%
}

\begin{document}

\maketitle
%\thispagestyle{empty}
%\pagestyle{empty}

%%%%%%%%%%%%%%%%%%%%%%%%%%%%%%%%%%%%%%%%%%%%%%%%%%%%%%%%%%%%%%%%%%%%%%%%%%%%%%
\begin{abstract}
We deal with the convergence of the value function of an approximate control problem with uncertain dynamics to the value function of a nonlinear optimal control problem. The assumptions on the dynamics and the costs are rather general and we assume to represent uncertainty in the dynamics by a probability distribution. The proposed framework aims to describe and motivate some model-based Reinforcement Learning algorithms where the model is probabilistic. We also show some numerical experiments which confirm the theoretical results.
\end{abstract}

%%%%%%%%%%%%%%%%%%%%%%%%%%%%%%%%%%%%%%%%%%%%%%%%%%%%%%%%%%%%%%%%%%%%%%%%%%%%%%
\section{Introduction}

Reinforcement Learning (RL) is an important branch of Machine Learning aiming to provide satisfactory policies that an agent can easily implement in an uncertain environment \cite{SuttonBarto}.  The agent acquires knowledge (\textit{learns}) on the environment from its past experience, which is usually represented by series of statistical data, but can also learn while interacting with the system. Optimal control \cite{Fleming-Rishel,Bardi-Capuzzo} and RL are strongly connected \cite{sutton1992,recht2019}, to the point that in \cite{sutton1992} the authors write \emph{``Reinforcement Learning is direct adaptive optimal control''}.

RL algorithms are generally classified in two categories: \textit{model-based} methods, which build a predictive model of the environment and use it to construct a controller, and \textit{model-free} methods, which directly learn a policy or a value function by interacting with the environment. Model-free algorithms have shown great performances \cite{Mnih2015,TRPO2015,Lillicrap2016,SAC2018}, although they are generally quite expensive to train, especially in terms of sample complexity; this often limits their applications to simulated domains. On the contrary, model-based techniques show a higher sample efficiency \cite{PILCO2013,BenchmarkMBRL2019}, which results in faster learning. Furthermore, recent algorithms have managed to limit the model-bias phenomenon by using \textit{probabilistic models}, which capture the uncertainty of the learned model \cite{PILCO2013,DeepPILCO,Kamthe2018}. In Bayesian Reinforcement Learning (BRL), the dynamics model is updated when new data are available \cite{BRL-Survey}.
Finally, the recently used probabilistic model ensembles \cite{PETS2018,MBPO2019} allowed model-based methods to achieve the same asymptotic performance as state-of-the-art model-free methods, with higher sample efficiency. Thanks to these features, model-based methods seem to be the most suitable for solving complex real-world problems.

In this paper, we consider the class of BRL algorithms for continuous state-action space and we analyze them from the viewpoint of control theory. In particular, we consider nonlinear control systems in which the dynamics is partially known and we assume that the belief on the dynamics that an agent has is represented by a probability distribution $\pi$ on a space of functions \cite{murray2018model}.
The task is thus written as an optimal control problem with averaged cost, a kind of formulation that has received a growing interest in the last few years (see e.g. \cite{Bettiol-Khalil,zuazua2014averaged,loheac2016}). In probabilistic model-based RL algorithms, this corresponds to the \textit{policy improvement} step, where the agent seeks to find the best control given a probabilistic model $\pi$, using a policy search method \cite{PILCO2013,DeepPILCO,MBPO2019} or MPC \cite{Kamthe2018,PETS2018} or other methods.
In the framework we propose, the precision with which $\pi$ describes the true dynamics improves as soon as the dataset becomes wider. This reflects a situation in BRL where the agent learns on the surrounding environment as far as he gets more experience and \textit{updates} the dynamics model.

The main objective of the paper concerns a convergence result of the value function $V_\pi$ of an averaged (with respect to a probability measure $\pi$)  optimal control problem to the ``true value function" $V$. Here, by true value function we mean the one defined by the optimal control governed by the \textit{true}, underlying dynamics. Roughly speaking, the main result of the paper can be stated as follows: \textit{the value function $V_\pi$ is close to $V$ as soon as $\pi$ provides an accurate representation of the true, underlying dynamics}. Similar results have been recently obtained for a general Linear Quadratic Regulator problem with finite horizon (see \cite{pesare2020LQR}). In the present paper, we will focus our attention on a general, nonlinear optimal control problem over a finite horizon, under globally Lipschitz assumptions on the costs and the controlled dynamics.

In the next section, we present the precise setting of the paper whereas in section~\ref{sec: results} we will state and prove the main result. Section~\ref{sec: numerics} is devoted to some numerical tests and section~\ref{sec: conclusions} provides the conclusions and an overview on open questions.

\section{Problem formulation}
This section aims to propose the nonlinear optimal control framework that we want to solve. Before that, let us introduce some notations which will be used throughout the paper. For vectors $v \in \mathbb{R}^n$, $|v|$ will denote the Euclidean norm.
For continuous functions $f \colon D \to \mathbb{R}$ with $D \subset \mathbb{R}^n$, the notations $\norm{f}_\infty$ and $\norm{f}_{\infty,K}$ will indicate the sup norm respectively over the function domain $D$ or over a compact set $K \subset D$. \\
%For a generic metric space $(X,d)$, $\mathcal{M}(X)$ will denote the space of Borel probability measures on $X$, equipped with the weak-$^\star$ topology. When the space $X$ is compact, the weak-$^\star$ topology is metrized by the Wasserstein distance (see e.g. \cite{Villani})
Let $\pi$ and $\pi'$ be two probability distributions on a compact metric space $(X,d)$. The 1-Wasserstein distance (see \cite{Villani}) between them is defined as
\begin{equation}\label{wasserstein_distance}
	W_1(\pi,\pi') \coloneqq \inf_{\gamma \in \Gamma(\pi,\pi')}{\int_{X \times X} \norm{g-f}_\infty \, d\gamma(g,f)}  \ ,
\end{equation}
where $\Gamma(\pi,\pi')$ is the collection of all probability measures on $X \times X$ having $\pi$ and $\pi'$ as marginals, $f$ and $g$ are generic elements in $X$ and the symbol $d\gamma$ indicates that the integral is with respect to the measure $\gamma$.\\
Given a probability space $(\Omega,\mathcal{F},\pi)$  and a random variable $Y$ on $\Omega$, $\mathbb{E}_\pi[Y]$ denotes the expected value of $Y$ with respect to $\pi$.

\subsection{Problem A: a classical optimal control problem}\label{PA}
Let us consider a classical finite horizon optimal control problem (\cite{Fleming-Rishel,Bardi-Capuzzo}), which we will call \emph{Problem A}. For \mbox{$0\leq s< T$}, let us consider the controlled dynamics
\begin{equation}\label{TD}
	\left\{ 
	\begin{aligned}
		\dot{x}(t)&=f(x(t),u(t))  \qquad \qquad t \in [s,T] \\
		x(s)&=x_0 , 
	\end{aligned}
	\right.
\end{equation}
where the nonlinear dynamics $f:\mathbb{R}^n\times \mathbb{R}^m\rightarrow \mathbb{R}^n$ is continuous in the pair $(x,u)$ and Lipschitz continuous with respect to $x$, uniformly with respect to $u$. Those conditions guarantee that the Cauchy problem \eqref{TD} is well-posed, in the sense that for every measurable control $u$ and initial condition $x(s)=x_0\in \mathbb{R}^n$, there exists a unique solution of \eqref{TD}.

The goal is to minimize the cost functional
\begin{equation}\label{cost-A} J_{s, x_0}[u]\coloneqq \int_s^T \ell (x(t),u(t)) dt + h(x(T))  \ , \end{equation}
over the class of the \emph{admissible controls} $\mathcal{U}_s=\{u:[s,T]\rightarrow U, \ \textrm{measurable} \}$, where $U$ is a closed subset of $\mathbb{R}^m$ and $\ell$ and $h$ are respectively the running cost and the terminal cost, which we require to be Lipschitz continuous with respect to $x$. For each $(s, x_0)\in[0,T]\times \mathbb{R}^n$, the \emph{value function}  and the corresponding {\em optimal control} associated to the optimal control problem \eqref{TD}-\eqref{cost-A} are respectively defined as
\[ V(s, x_0) \coloneqq \inf\limits_{u \in \mathcal{U}_s} J_{s,x_0}[u] \quad \hbox{and} \quad u^*(s,x_0) \coloneqq {\rm arg}\min\limits_{u \in \mathcal{U}_s} J_{s, x_0}[u] \, . \]
\subsection{Problem B: an optimal control problem with uncertain dynamics}\label{PB}
Let us now introduce another control problem in which the real  dynamics $f$ is \textit{unknown}, meaning that one has merely a partial knowledge on $f$. Such a model uncertainty is captured by a probability distribution on a space of functions $X$  (which $f$ belongs to). More precisely, $X$ is a compact subset of $C^0(\mathbb{R}^n\times U; \mathbb{R}^n)$ with respect to the $||\cdot||_{\infty}$ norm (the Arzel\`a-Ascoli Theorem provides necessary and sufficient conditions for the set $X$ being compact). For $0\leq s< T$ and every $g \in X$, let us define the dynamical system
\begin{equation}\label{UD}
	\left\{ 
	\begin{aligned}
		\dot{x}^g(t)&=g\left( x^g(t),u(t) \right) \qquad \qquad t \in [s,T] \\
		x^g(s)&=x_0 .
	\end{aligned}
	\right.
\end{equation}
Given a probability distribution $\pi$ over $X$, one can define a cost functional for Problem B:
\begin{align}
	J_{\pi,s,x_0} [u] \coloneqq& \mathbb{E}_\pi \left[ \int_{s}^T \ell(x^g(t),u(t)) dt + h(x(T)) \right] \label{cost_B} \\
	=& \int_X \left[ \int_{s}^T \ell(x^g(t),u(t)) \, dt + h(x(T)) \right] \, d\pi(g) . \notag
\end{align}
Note that, even if there is a different trajectory $x^g(t)$  for each $g \in X$, the task concerns to look for a single control $u$ to be applied to every system dynamics $g \in X$. \\
The notions of value function and optimal control are analogous to those given for Problem A:
\begin{equation*}
	V_\pi(s, x_0) \coloneqq \inf_{u \in \mathcal{U}_s} J_{\pi,s,x_0}[u]  \ \hbox{ and } \ 
	u_\pi^*(s,x_0) \coloneqq {\rm arg}\min_{u \in \mathcal{U}_s} J_{\pi,s,x_0}[u] \, . 
\end{equation*}

\begin{remark}\label{rmk: parameters}
	It is worth pointing out that the theory developed here works both for non-parametric models (e.g. Gaussian processes \cite{PILCO2013,Kamthe2018}) and for parametric models (e.g. deep neural networks (\cite{DeepPILCO,PETS2018,MBPO2019})). In the latter case, we are assuming that the support of $\pi$ is a family of functions $\{f_\lambda\}_{\lambda \in \mathbb{R}^d}$ described by a parameter $\lambda \in \mathbb{R}^d$, with the dimension $d$ arbitrarily large. $\pi$ can thus be seen as a probability distribution on the parameter space $\mathbb{R}^d$ and is then easier to work with (see, for instance, the numerical tests in section~\ref{sec: numerics}).
\end{remark}

There are several model-based RL methods that rely on the design of probability distributions representing the belief that an agent has on the environment. In BRL algorithms \cite{PILCO2013,DeepPILCO,Kamthe2018,PETS2018,MBPO2019,BRL-Survey}, the probability distribution, which is built upon the data collected while exploring the partially unknown environment, is updated when new experience is gained. In this context, it is reasonable to expect that the newest probability distribution representing the environment will be more accurate than the initial one.

In the framework of our paper, $\pi$ is the probability distribution, i.e. the probabilistic model, which represents the knowledge that the agent has on the environment. Clearly, as the accuracy of $\pi$ increases, one should expect that the value function of Problem B is close (in a sense that will be made precise) to the value function of Problem A. In particular, we will  investigate the following questions:
\begin{enumerate}[label=(\Alph*)]
	\item How far is $V_\pi$ from $V$?
	\item If $\pi^N \to \delta_f$, then is it true that $V_{\pi^N} \to V$, where $\{ V_{\pi^N} \}_{N \in \mathbb{N}}$ are the value functions of a sequence of problems of type B, and $\{ \pi^N \}_{N \in \mathbb{N}}$ are the respective probability distributions on $X$?
\end{enumerate}

\section{Main Results}\label{sec: results}
In this section, we will state and prove the main result of the paper valid for a general family of probability distributions.
\begin{theorem}\label{General_thm}
	Let us consider two Problems of type B as described in section~\ref{PB}, one with distribution $\pi^N$ and the other with distribution $\pi^\infty$. We make the further assumptions:
	\begin{enumerate}[label=(\roman*)]
		\item[$(H1)$] There exists a constant $L_f>0$ such that
		\[ |f(x,u)-f(y,u)| \leq L_f |x-y| \]
		for each $f \in supp(\pi^\infty)$, $x,y \in \mathbb{R}^n$, $u \in U$;
		\item[$(H2)$] The two cost functions $\ell$ and $h$ are Lipschitz continuous in the first argument with constants respectively $L_\ell$ and $L_h$.
	\end{enumerate}
	Then the following estimate holds:
	\begin{equation}\label{Vgen_est}
		||V_{\pi^N} - V_{\pi^\infty}||_\infty \leq C(L_f, L_\ell, L_h, T) \, W_1(\pi^N,\pi^\infty) \ ,
	\end{equation}
	where $C(L_f, L_\ell, L_h, T)$ is a constant which depends only on the Lipschitz constants and on $T$, and $W_1(\pi^N,\pi^\infty)$ is the 1-Wasserstein distance defined in \eqref{wasserstein_distance}.
\end{theorem}
\begin{proof} We divide the proof in three steps. \\
	\emph{STEP 1:} Fix two dynamics $g \in X$ and $f \in supp(\pi^\infty)$, an initial condition $x(s) = x_0$ with $s \in [0,T]$ and $x_0 \in \mathbb{R}^n$ and a control $u \in \mathcal{U}_s$. We estimate how far is $x^g(t)$ from $x^f(t)$, using Gronwall's Lemma, for each $t\in [s,T]$. \\
	Recall that $x^f(t)$ and $x^g(t)$ are solutions of the dynamical systems \eqref{UD}:
	\begin{equation*}
		\begin{aligned}
			x^f(t) &= x_0 + \int_s^t f(x^f(\tau),u(\tau)) \, d\tau \\
			x^g(t) &= x_0 + \int_s^t g(x^g(\tau),u(\tau)) \, d\tau
		\end{aligned}
	\end{equation*}
	Then we have the following estimate:
	\begin{equation*}
		\begin{aligned}
			|x^g(t)-x^f(t)| &\leq \int_s^t |g(x^g(\tau),u(\tau))-f(x^f(\tau),u(\tau))| \, d\tau \\
			&\leq \int_s^t |g(x^g(\tau),u(\tau))-f(x^g(\tau),u(\tau))| \, d\tau \\
			& \quad \ + \int_s^t |f(x^g(\tau),u(\tau))-f(x^f(\tau),u(\tau))| \, d\tau \\
			&\leq (t-s) \, \norm{f-g}_\infty + L_f \int_s^t |x^g(\tau)-x^f(\tau)| \, d\tau \, .
		\end{aligned}
	\end{equation*}
	Then, by Gronwall's Lemma, 
	\begin{equation}\label{x_est}
		|x^g(t)-x^f(t)| \leq (t-s) \, \norm{f-g}_\infty \, e^{L_f (t-s)}\leq t \, \norm{f-g}_\infty \, e^{L_f t} .
	\end{equation}
	\emph{STEP 2:} Fix an initial condition $x(s) = x_0$ with $s \in [0,T]$ and $x_0 \in \mathbb{R}^n$ and a control $u \in \mathcal{U}_s$. We estimate how far is $J_{\pi^N,s,x_0}[u]$ from $J_{\pi^\infty,s,x_0}[u]$. To enlighten the notation, we will write
	\[ J_{\pi^N}[u]=J_{\pi^N,s,x_0}[u] \quad \textrm{and} \quad J_{\pi^\infty}[u]=J_{\pi^\infty,s,x_0}[u] . \]
	For each $g \in X$ and $f \in supp(\pi^\infty)$ it holds
	\begin{equation} \begin{aligned} \label{ell_bound}
			| l{\left(x^g(t),u(t) \right)}- l{\left(x^f(t),u(t) \right))} &\leq L_\ell \, |x^g(t)-x^f(t)| \\
			& \leq L_\ell \, t \, \norm{f-g}_\infty \, e^{L_f t} \end{aligned} \end{equation}
	and
	\begin{equation} \label{h_bound} 
		| h{\left(x^g(T)\right)}- h{\left(x(T)\right)}| \leq L_h \, T \, \norm{f-g}_\infty \, e^{L_f T} . 
	\end{equation}
	As a property of $W_1$, there exists (see Theorem~4.1 on \cite{Villani}) a distribution $\gamma^*$ on $X \times X$ with marginal distributions $\pi^N$ and $\pi^\infty$ such that
	\[ W_1(\pi^N,\pi^\infty) = \int_{X \times X} \norm{g-f}_\infty \, d\gamma^*(g,f) . \]
	Let ${(g(x),f(x))}_{x \in \mathbb{R}^n}$ be a continuous process which has $\gamma^*$ as distribution. For each event $\omega \in \Omega$, we consider the realization of the process $(g^\omega,f^\omega)$. \\
	Let us sum up over the $g \in supp(\pi^N)$ and the $f \in supp(\pi^\infty)$:
	\begin{equation*} \begin{aligned}
			J_{\pi^N}[u] &- J_{\pi^\infty}[u] = \\ 
			& = \int_X \biggl[ \int_s^T \ell(x^g(t),u(t)) \, dt + h(x^g(T))\biggr] d\pi^N(g) \\
			& \quad - \int_X \biggl[ \int_s^T \ell(x^f(t),u(t)) dt + h(x^f(T)) \biggr] d\pi^\infty(f)  \\
			& = \int_{XxX} \biggl[ \int_s^T \bigl( \ell(x^g(t),u(t)) - \ell(x^f(t),u(t)) \bigr) \, dt \, \\
			& \qquad \qquad + \bigl( h(x^g(T)) - h(x^f(T)) \bigr) \biggr] d\gamma^*(g,f) .
	\end{aligned} \end{equation*}
	Passing to the absolute value and using the bounds in \eqref{ell_bound} and \eqref{h_bound}, we recover the expression of $W_1$:
	\begin{equation}\label{J_est} \begin{aligned} 
			\Bigl| J_{\pi^N}[u] &- J_{\pi^\infty}[u] \Bigr| \leq \\ 
			&\leq \left[ L_\ell \int_0^T t \, e^{L_ft} dt + L_h \, e^{L_f T} \right] \, \int_{X \times X} \norm{g-f}_\infty \, d\gamma(g,f) \\
			&= \left( L_\ell \ \frac{e^{L_f T} (L_f T - 1) + 1}{{L_f}^2} + L_h e^{L_f T} \right) \, W_1(\pi^N,\pi^\infty)  \\
			& = C(L_f,L_\ell,L_h,T) \, W_1(\pi^N,\pi^\infty) .
	\end{aligned} \end{equation}
	Note that this estimate does not depend on $x_0$ or $u$. \\
	\emph{STEP 3:} We now prove  the estimate \eqref{Vgen_est} using \eqref{J_est}. Fix an initial condition $x(s)=x_0$ and some $\eps > 0$. By the definition of $V_{\pi^\infty}$, there exists some control $u \in \mathcal{U}_s$ such that
	\[ J_{\pi^\infty}[u_\eps] \leq V_{\pi^\infty}(s, x_0) + \eps \, . \]
	Then one has
	\begin{equation*}
		\begin{aligned}
			V_{\pi^N}(s, x_0) - V_{\pi^\infty}(s, x_0) &= \inf_{u \in \mathcal{U}_s}{J_{\pi^N}[u]} - \inf_{u \in \mathcal{U}_s}{J_{\pi^\infty}[u]} \\
			& <\inf_{u \in \mathcal{U}_s}{J_{\pi^N}[u]} - J_{\pi^\infty}[u_\eps] + \eps  \\
			&\leq J_{\pi^N}[u_\eps]- J_{\pi^\infty}[u_\eps] + \eps \\
			&\leq \sup_{u \in \mathcal{U}_s}\left| J_{\pi^N} - J_{\pi^\infty} \right| + \eps .
		\end{aligned}
	\end{equation*}
	In the same way, we get 
	\[ V_{\pi^\infty}(s, x_0) - V_{\pi^N}(s, x_0)  <  \sup_{u \in \mathcal{U}_s}\left| J_{\pi^N} - J_{\pi^\infty} \right| + \eps \ , \]
	but the $\eps$ is arbitrary, so eventually
	\[ \left| V_{\pi^N}(s, x_0) - V_{\pi^\infty}(s, x_0) \right| \leq \sup_{u \in \mathcal{U}}\left| J_{\pi^N} - J_{\pi^\infty}  \right| \ . \]
	Finally, noting that the estimate is independent from $x_0$, we get the result:
	\begin{align*}
		|| V_{\pi^N} - V_{\pi^\infty} ||_\infty &\leq \sup_{u \in \mathcal{U}}\left| J_{\pi^N} - J_{\pi^\infty}  \right| \\
		&\leq C(L_f, L_\ell, L_h, T) \,W_1(\pi^N,\pi^\infty) .
	\end{align*}
\end{proof}
In the particular case where $\pi^\infty\equiv \delta_f$, $f$ being the true underlying dynamics, then one can express Theorem \ref{General_thm} in a more expressive fashion as follows:
\begin{corollary}[Convergence of the value functions]\label{cor:conv} Suppose that $f:\mathbb{R}^n\times U\rightarrow \mathbb{R}^n$ is a Lipschitz continuous function and the assumption $(H2)$ on $\ell$ and $h$ is satisfied.   Let $\{\pi^N\}_{N \in \mathbb{N}}$ be a sequence of probability distributions on \mbox{$X\subset C_0(\mathbb{R}^n\times U$)} compact set, with $\pi^N \!\! \xrightarrow{\: W_1 \: } \delta_f$. Then the value function $V_{\pi^N}$ of Problem B converges uniformly  on  $ [0,T]\times \mathbb{R}^n$ to the value function $V$ of Problem A.  
	%	\pes{\\ Oppure: If $f$ is Lipschitz continuous, then $V_{\pi^N}$ converges to $V$, uniformly for $x_0$ in compact sets, where $V$ is the value function of Problem A relative to the true dynamics $f$.}
\end{corollary}

The previous Corollary provides a positive answer to the questions (A)-(B) posed in the introduction. More precisely, it tells us that, whenever it is possible to construct a sufficiently close (with respect to the Wasserstein distance) probability distribution $\pi$ to the real dynamics $f$, then the value function $V_{\pi}$ is a good approximation of the value function $V$.

\subsection{A case of study: finite support measures converging to $\delta_f$}\label{sec: finite}
Let us consider the case in which the distribution $\pi^N$ is a linear combination of a finite number of Dirac deltas defined on a family of equi-bounded and equi-Lipschitz continuous functions $X:=\left\{ f_1, \dots, f_M\right\}$:
\begin{equation}\label{discrete_pi} \pi^N := \sum_{i=1}^M \alpha_i^N \, \delta_{f_i}, \end{equation}
where $\alpha_i^N\geq 0$, $\sum_{i=1}^M \alpha_i^N=1$ for every $i=1,\ldots,M$ and $N\in \mathbb{N}$.
In this case, the cost functional \eqref{cost_B} can be written as
\[ J_{\pi^N,s,x0}[u] := \sum_{i=1}^M \alpha_i^N \left[ \int_s^T \ell(x^{f_i}(t),u(t)) \, dt + h(x^{f_i}(T)) \right] . \]
Without loss of generality, let us also assume that $f\equiv f_1$ is the real underlying dynamics. Then it follows from Theorem \ref{General_thm}  the next result:
\begin{corollary}\label{cor:finite} Let us consider a sequence of probability distributions $\pi^N$ defined as in \eqref{discrete_pi}. Assume that \mbox{$f_1:\mathbb{R}^n\times U\rightarrow \mathbb{R}^n$} is Lipschitz continuous and assume the assumption $(H2)$ on the functions $\ell$ and $h$.
	%		\begin{enumerate}[label=(\roman*)]
	%		\item The true dynamics $f$ is Lipschitz continuous in the first argument (uniformly with respect to the second argument) with constant $L_f$;
	%		\item The two cost functions $\ell$ and $h$ are Lipschitz continuous in the first argument (uniformly with respect to the second argument) with constants $L_\ell$ and $L_h$.
	%	\end{enumerate}
	Then
	\begin{enumerate}
		\item $ ||V_{\pi^N}- V||_\infty \leq C(L_f, L_\ell, L_h, T) \ \mathbb{E}_{\pi^N}{\left[ \norm{f-g}_\infty \right]}$ $\forall \, N \in \mathbb{N}$,
		where $V$ is the value function of Problem A relative to the true dynamics $f_1$.
		\item If $\{ \pi^N \}_{N \in \mathbb{N}}$ is a sequence of probability distributions with the same support $\{f_1, \dots, f_M\}$ and $\pi^N$ is converging to $\delta_{f_1}$, then $V_{\pi^N}$ converges to $V$.
	\end{enumerate}
\end{corollary}

%	\begin{proof}  
%		The thesis follows straightly from Theorem~\ref{General_thm}, noting that the quantity
%		\[ \mathbb{E}_{\pi^N}{\left[ \norm{f-g}_\infty \right]} \]
%		coincides with the 1-Wasserstein Distance between $\pi^N$ and $\delta_f$:
%		\[ W_1(\pi^N,\delta_f) := \inf_{\gamma \in \Gamma(\pi^N,\delta_f)}{\int_{X \times X} \norm{g_1-g_2}_\infty \, d\gamma(g_1,g_2)}  \ , \]
%		where $\Gamma(\pi^N,\delta_f)$ denotes the set of all joint probability distributions on $X \times X$, having $\pi^N$ and $\delta_f$ as marginals. In this case, indeed, the set $\Gamma(\pi^N,\delta_f)$ consists of a single element
%		\[ \Gamma(\pi^N,\delta_f) = \left\{ \sum_{i=1}^M \alpha_i \, \delta_{(f_i,f)}  \right\} \ , \]
%		and we get
%		\[ W_1(\pi^N,\delta_f) = \sum_{i=1}^M \alpha_i \, d_\infty(f_i,f) = \mathbb{E}_{\pi^N}{\left[ \norm{f-g}_\infty \right]} . \]
%	\end{proof}

\section{Numerical Tests}\label{sec: numerics}
In this section, we will present two numerical tests. Both tests deal with an extremely simplified situation: a parametric model, where the parameter can take only a finite number of values (cf. section \ref{sec: finite}). We remark that these tests are intended for \textit{illustrative purposes} only and their main goal is to verify that Theorem~\ref{General_thm} and Corollary~\ref{cor:conv} hold in a particular case. Indeed, the theory developed here includes far more general cases than this, including parametric models with a large number of parameters and infinite possible values for each parameter (e.g. deep neural networks \cite{DeepPILCO,PETS2018,MBPO2019}) or even non-parametric models (e.g. Gaussian processes \cite{PILCO2013,Kamthe2018}).

\subsection{Test 1}
We consider a dynamical system governed by the differential equation
\begin{equation*}
	\left\{ 
	\begin{aligned}
		\dot{x}(t) &= \lambda x(t) + sin(x(t)) + u(t) \qquad  t \in [s,T] \\
		x(s)&=x_0 , 
	\end{aligned}
	\right.
\end{equation*}
with $u(t) \in U = [-1,1]$ for $t \in [s,T]$.
The agent, who doesn't know the parameter $\lam$, has a probability distribution on a set of 5 possible values for $\lam$:
\[ \left\{ \lam_1=0, \ \lam_2=1, \ \lam_3=-1, \ \lam_4=0.5, \ \lam_5=-0.5  \right\} . \]
For each $N \in \mathbb{N}$ the probability distribution can be written, similarly to \eqref{discrete_pi}, as
\[ \pi^N = \sum_{i=1}^5 \alpha_i^N \delta_{\lambda_i} . \]
Note that a probability distribution on a set of parameters is equivalent to a probability distribution on a family of functions $\{f_i\}_i$ (see also Remark~\ref{rmk: parameters}), where
\[ f_i(x,u) = \lambda_i x + sin(x) + u . \]
We set $s \equiv 0$. The agent needs to minimize an averaged cost
\begin{equation}\label{test1: cost}
	\begin{aligned}
		J_{\pi^N,0,x_0}[u]  &= \sum_{i=1}^5 \alpha_i^N \left[ \int_0^T u(t)^2 dt - x_i(T) \right] \\
		&= \int_0^T u(t)^2 dt + \sum_{i=1}^5 \alpha_i^N \left[  - x_i(T) \right]
	\end{aligned}
\end{equation}
over the set of measurable controls $\mathcal{U}$. The Problem B associated to a fixed $\pi^N$ can thus be seen as an optimal control problem in dimension 5, where the state variable include the five trajectories $x_1, \dots, x_5$.

We solved this problem numerically when the probabilities $\alpha_i^N$ were defined according to the following rule:
\[ \alpha_1^N = 1-\frac{1}{2^ N} \, , \quad \alpha_i^N = \frac{1}{4} \frac{1}{2^N} \ \textrm{for } i=2,\dots,5 \ .  \]
It is clear that the sequence is converging to $\delta_{\lambda_1}$, which we assume to be the parameter $\lam$ of the \emph{true dynamics}, following the setting of subsection~\ref{sec: finite}. 
An example of optimal (multi-)trajectory of Problem B relative to $\pi^1$ is plotted in Fig.~\ref{fig:test1_traj}. In order to minimize the cost functional \eqref{test1: cost}, the agent tries to steer all the trajectories towards the positive values of the real axis, using a control close to $+1$; at the same time, the optimal control cannot be constantly $+1$, since the cost functional penalizes larger values of the control (Fig.~\ref{fig:test1_control}).

\begin{figure}[tbp]
	\centering
	\includegraphics[scale=0.43]{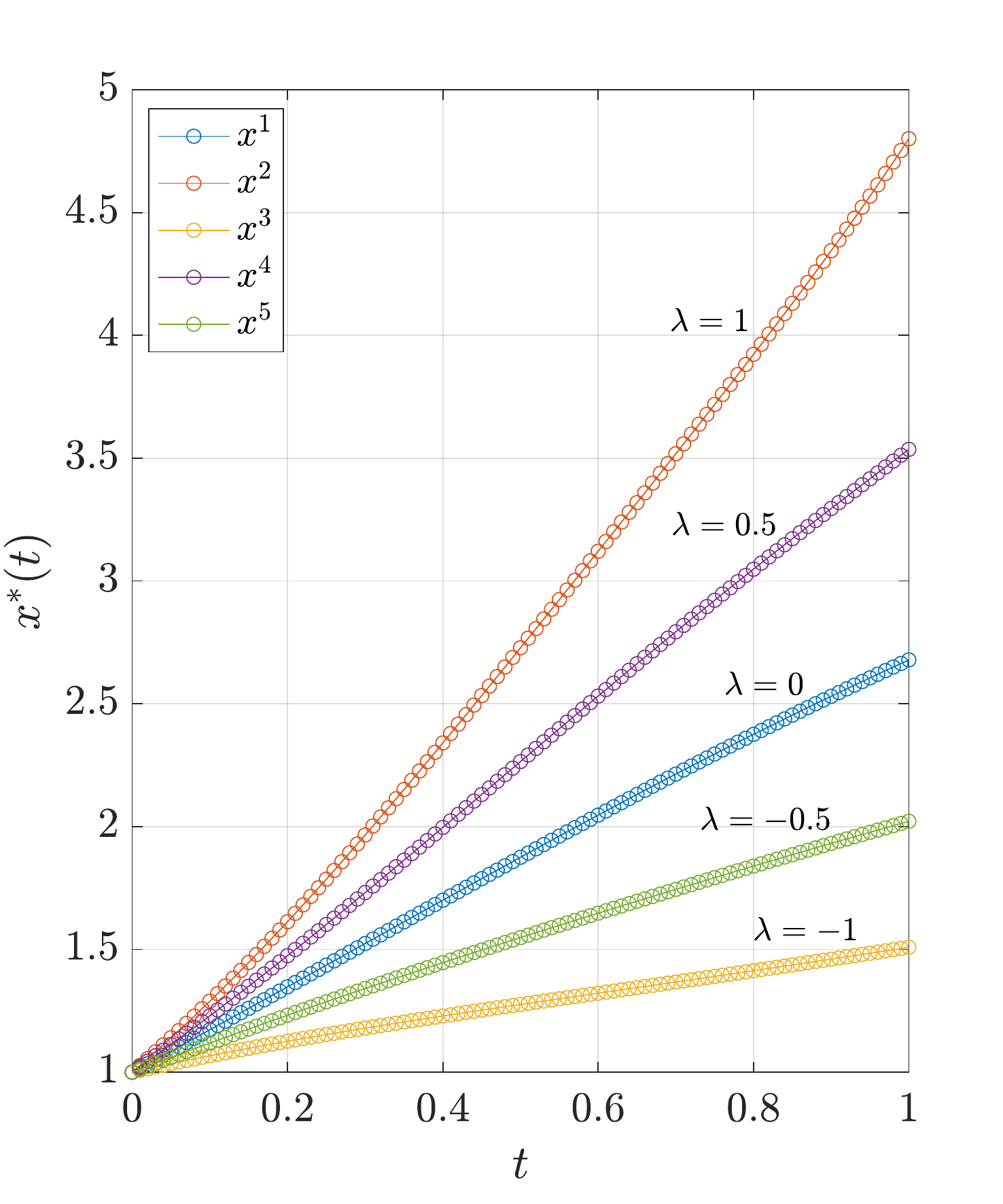}
	\caption{Test 1: Optimal (multi-)trajectory starting from $x_0 = 1$.}
	\label{fig:test1_traj}
\end{figure}

\begin{figure}[tbp]
	\centering
	\includegraphics[scale=0.43]{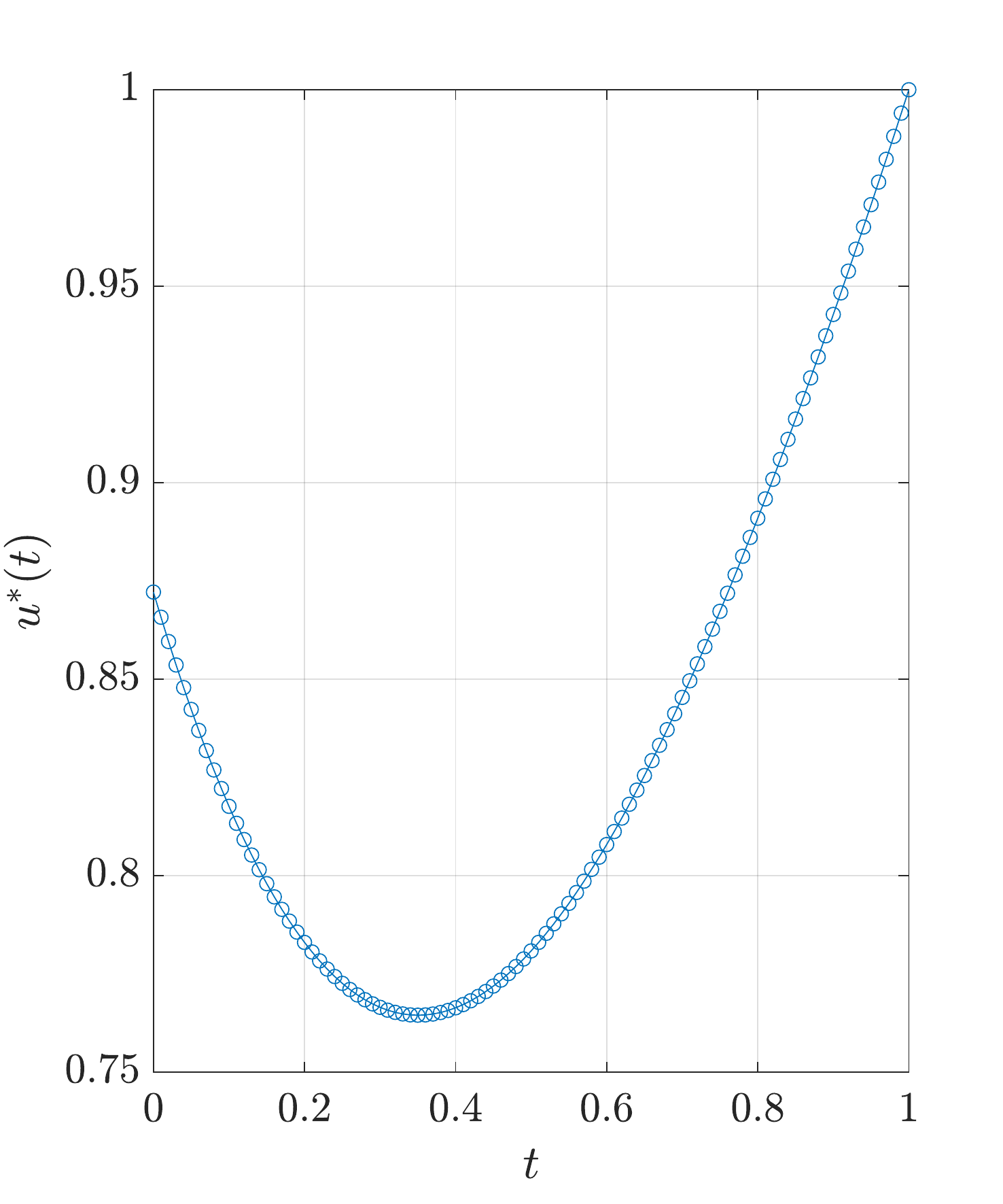}
	\caption{Test 1: Optimal control starting from $x_0 = 1$.}
	\label{fig:test1_control}
\end{figure}

For each $N=1, \dots, 8$ we computed the value function of Problem B solving the equations given by the Pontryagin Maximum Principle (see e.g. \cite{Fleming-Rishel}), for a grid of initial points $x_0 \in [-1,1]$. Then, we compared them to the value function of Problem A relative to the true dynamics $f_1$, computing the sup norm of the difference $V_{\pi^N}-V$ over the interval $[-1,1]$; the results are reported in Table~\ref{tab:test1}. Note that at each iteration the Wasserstein distance $W_1$ between $\pi^N$ and $\delta_{f_1}$ is halved and so is the error; this means that the numerical convergence order is 1, which agrees to the estimate given by Corollary~\ref{cor:finite}.

\begin{table}[tbhp]
	\centering
	\def\arraystretch{1.5}%  1 is the default, change whatever you need
	\begin{tabular}{|c|c|c|c|}
		\hline
		$N$  	 &	$\alpha_1^N$		& $||V_{\pi^N}-V||_{\infty, [-1,1]}$	& $order$ \\
		\hline
		1	 & 0.5  	  & 1.57e-1    	  & -	  		\\
		\hline
		2	 & 0.75  	 & 7.87e-2    	 & 1.00	     \\
		\hline
		3	 & 0.875  	& 3.94e-2    	& 1.00		\\
		\hline
		4	 & 0.9375  & 1.97e-2    	& 1.00		\\
		\hline
		5	 & 0.9687  & 9.84e-3    	& 1.00		\\
		\hline
		6	 & 0.9844  & 4.92e-3    	& 1.00		\\
		\hline
		7	 & 0.9922  & 2.46e-3    	& 1.00		\\
		\hline		
		8	 & 0.9961  & 1.23e-3    	& 1.00		\\
		\hline
	\end{tabular}
	\caption{Test 1: Errors for value functions related to $\pi^N$ for $N=1, \dots, 8$ with respect to the true value function of Problem A.}
	\label{tab:test1} 	
\end{table}

\subsection{Test 2}
The convergence results presented in this work hold under the assumption that the cost functions $\ell$ and $h$ are both globally Lipschitz continuous. With the following numerical test, we show that there are other examples of practical interest, where we can observe similar convergent behavior in the error $\norm{V_{\pi^N}-V}_\infty$, even though this hypothesis is not verified.

In this second example, the state of the system is \mbox{2-dimensional}. The three possible dynamics are all linear in the space variable and they differ only by the system matrix $A_i \in \mathbb{R}^{2 \times 2}$:
\[ 
\begin{pmatrix} \dot{x}_1^i \\ \dot{x}_2^i \end{pmatrix} 
= A_i \begin{pmatrix} x_1^i \\ x_2^i \end{pmatrix} 
+ \begin{pmatrix} cos(u) \\ sin(u) \end{pmatrix}  , \]
where $u$ is a 1-dimensional control which lies in $[0,2\pi]$.

The three possible matrices are
\[ A_1 = \begin{pmatrix} 1 & 0  \\ 0 & 1 \end{pmatrix}  , \
A_2 = \begin{pmatrix} 0.5 & 0 \\ 0 & 2 \end{pmatrix} \textrm{ and }
A_3 = \begin{pmatrix} 0.5 & -0.5 \\ 0.5 & 0.5 \end{pmatrix} . \]
Without loss of generality, we assume that the true dynamics corresponds to the first matrix. The agent only knows a probability distribution $\pi^N$ on the set of the three matrices, defined as in the previous example:
\[ \pi^N = \sum_{i=1}^3 \alpha_i^N \delta_{A_i} , \]
where the weights are defined according to the rule
\[ \alpha_1^N = 1-\frac{1}{2^ N} \, , \quad \alpha_i^N = \frac{1}{2^{N+1}} \ \textrm{ for } i=2,3  .  \]

We set again $s \equiv 0$. The functional cost to be minimized is
\begin{equation}
	J_{\pi^N,0,x_0}[u] = \frac{1}{2} \sum_{i=1}^3 \alpha_i^N \left[ \int_0^T \norm{x^i(t)}^2 dt + \norm{x^i(T)}^2 \right] ,
\end{equation}
where $\norm{\cdot}$ indicates the euclidean norm in $\mathbb{R}^2$.
In Fig.~\ref{fig:test3_traj} we can see an example of optimal \mbox{(multi-)}trajectory for this problem, when the distribution $\pi$ has been chosen to be $\left\{\frac{1}{3},\frac{1}{3},\frac{1}{3}\right\}$. In this case, the agent has to minimize the average squared distance of the trajectories from the origin, thus he looks for a single control to steer all three trajectories towards the origin at the same time. Clearly, we can see that none of the trajectories reaches exactly the origin, since a control which is optimal for one of the three dynamics may not be optimal for the other two. 
\begin{figure}[tbp]
	\centering
	\includegraphics[scale=0.44]{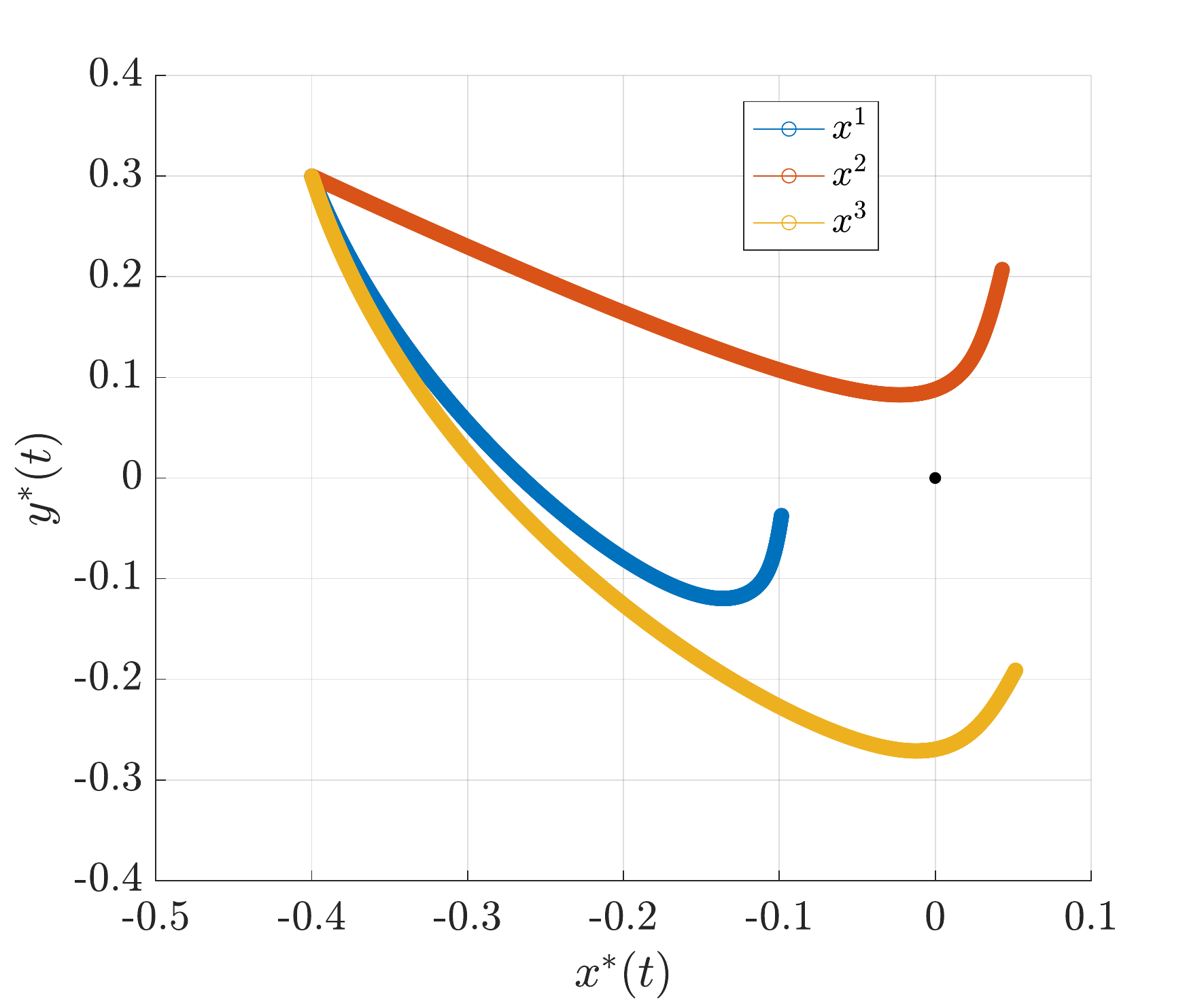}
	\caption{Test 2: Optimal (multi-)trajectory for $\pi=\left\{\frac{1}{3},\frac{1}{3},\frac{1}{3}\right\}$ starting from $x_0 = (-0.4,0.3)$.}
	\label{fig:test3_traj}
\end{figure}
In Table~\ref{tab:test3} the distance between the value function of Problem B and the true value function of Problem A is reported, for different distributions $\pi^N$. Even in this case, although the hypothesis of Corollary~\ref{cor:finite} are not satisfied, we observe a clear convergence with order 1.

\begin{table}[tbh]
	\centering
	\def\arraystretch{1.5}%  1 is the default, change whatever you need
	\begin{tabular}{|c|c|c|c|c|c|c|c|}
		\hline
		$N$  	 &	$\alpha_1^N$		& $||V_{\pi^N}-V||_{\infty, [-1,1]^2}$	& $order$ \\
		\hline
		1	 & 0.5  	  & 2.52e-0      	& -	  		\\
		\hline
		2	 & 0.75  	 & 1.32e-0       & 0.94	     \\
		\hline	
		3	 & 0.875  	& 6.81e-1    	& 0.95		\\
		\hline
		4	 & 0.9375  & 3.47e-1    	& 0.97		\\
		\hline
		5	 & 0.9687  & 1.75e-1    	& 0.98		\\
		\hline
		6	 & 0.9844  & 8.81e-2    	& 0.99		\\
		\hline
	\end{tabular}
	\caption{Test 2: Errors for value functions related to $\pi^N$ with $N=1, \dots, 6$ with respect to the true value function of Problem A.}
	\label{tab:test3}
\end{table}

\section{Conclusions}\label{sec: conclusions}

In this paper, we have shown some convergence properties of the value function for optimal control problems in an uncertain environment. In the framework of the paper, the degree of uncertainty of the control system is captured by a probability measure defined on a compact space of functions. Furthermore, such a probability measure is updated as soon as more information on the environment is gained. The paper framework is closely related to many model-based RL algorithms which aim at designing a suitable probability distribution rather than providing a pointwise estimate of the dynamics.

Similar results have been proved for the Linear Quadratic Regulator problem \cite{pesare2020LQR}. The main novelty of the paper consists in the assumptions on the dynamical system and the cost since we abandon the classical linear quadratic setting to deal with general nonlinear optimal control problems. We believe that the hypotheses of our theoretical result can be further relaxed. Some numerical examples seem to be a good omen for future investigations in this direction.
	
\bibliographystyle{IEEEtran}
\bibliography{ECC2021}

% Generated by IEEEtran.bst, version: 1.14 (2015/08/26)
\begin{thebibliography}{10}
\providecommand{\url}[1]{#1}
\csname url@samestyle\endcsname
\providecommand{\newblock}{\relax}
\providecommand{\bibinfo}[2]{#2}
\providecommand{\BIBentrySTDinterwordspacing}{\spaceskip=0pt\relax}
\providecommand{\BIBentryALTinterwordstretchfactor}{4}
\providecommand{\BIBentryALTinterwordspacing}{\spaceskip=\fontdimen2\font plus
\BIBentryALTinterwordstretchfactor\fontdimen3\font minus
  \fontdimen4\font\relax}
\providecommand{\BIBforeignlanguage}[2]{{%
\expandafter\ifx\csname l@#1\endcsname\relax
\typeout{** WARNING: IEEEtran.bst: No hyphenation pattern has been}%
\typeout{** loaded for the language `#1'. Using the pattern for}%
\typeout{** the default language instead.}%
\else
\language=\csname l@#1\endcsname
\fi
#2}}
\providecommand{\BIBdecl}{\relax}
\BIBdecl

\bibitem{SuttonBarto}
R.~S. Sutton and A.~G. Barto, \emph{Reinforcement Learning: An Introduction},
  2nd~ed.\hskip 1em plus 0.5em minus 0.4em\relax Cambridge, MA: MIT Press,
  2018.

\bibitem{Fleming-Rishel}
W.~H. Fleming and R.~W. Rishel, \emph{Deterministic and stochastic optimal
  control}.\hskip 1em plus 0.5em minus 0.4em\relax Springer Science \& Business
  Media, 2012, vol.~1.

\bibitem{Bardi-Capuzzo}
M.~Bardi and I.~Capuzzo-Dolcetta, \emph{Optimal Control and Viscosity Solutions
  of Hamilton-Jacobi-Bellman Equations}.\hskip 1em plus 0.5em minus 0.4em\relax
  Birkhauser, 1997.

\bibitem{sutton1992}
R.~S. Sutton, A.~G. Barto, and R.~J. Williams, ``Reinforcement {L}earning is
  {D}irect {A}daptive {O}ptimal {C}ontrol,'' \emph{IEEE Control Systems},
  vol.~12, no.~2, pp. 19--22, 1992.

\bibitem{recht2019}
B.~Recht, ``A tour of reinforcement learning: The view from continuous
  control,'' \emph{Annual Review of Control, Robotics, and Autonomous Systems},
  vol.~2, pp. 253--279, 2019.

\bibitem{Mnih2015}
V.~Mnih, K.~Kavukcuoglu, D.~Silver, A.~A. Rusu, J.~Veness, M.~G. Bellemare,
  A.~Graves, M.~Riedmiller, A.~K. Fidjeland, G.~Ostrovski \emph{et~al.},
  ``Human-level control through deep reinforcement learning,'' \emph{Nature},
  vol. 518, no. 7540, pp. 529--533, 2015.

\bibitem{TRPO2015}
J.~Schulman, S.~Levine, P.~Abbeel, M.~Jordan, and P.~Moritz, ``Trust region
  policy optimization,'' in \emph{International conference on machine
  learning}.\hskip 1em plus 0.5em minus 0.4em\relax PMLR, 2015, pp. 1889--1897.

\bibitem{Lillicrap2016}
T.~P. Lillicrap, J.~J. Hunt, A.~Pritzel, N.~Heess, T.~Erez, Y.~Tassa,
  D.~Silver, and D.~Wierstra, ``Continuous control with deep reinforcement
  learning,'' in \emph{4th International Conference on Learning Representations
  (ICLR)}, 2016.

\bibitem{SAC2018}
T.~Haarnoja, A.~Zhou, P.~Abbeel, and S.~Levine, ``Soft actor-critic: Off-policy
  maximum entropy deep reinforcement learning with a stochastic actor,'' in
  \emph{International Conference on Machine Learning}.\hskip 1em plus 0.5em
  minus 0.4em\relax PMLR, 2018, pp. 1861--1870.

\bibitem{PILCO2013}
M.~P. Deisenroth, D.~Fox, and C.~E. Rasmussen, ``Gaussian processes for
  data-efficient learning in robotics and control,'' \emph{IEEE transactions on
  pattern analysis and machine intelligence}, vol.~37, no.~2, pp. 408--423,
  2013.

\bibitem{BenchmarkMBRL2019}
T.~Wang, X.~Bao, I.~Clavera, J.~Hoang, Y.~Wen, E.~Langlois, S.~Zhang, G.~Zhang,
  P.~Abbeel, and J.~Ba, ``Benchmarking model-based reinforcement learning,''
  \emph{arXiv preprint arXiv:1907.02057}, 2019.

\bibitem{DeepPILCO}
Y.~Gal, R.~McAllister, and C.~E. Rasmussen, ``Improving {PILCO} with {B}ayesian
  neural network dynamics models,'' in \emph{Data-Efficient Machine Learning
  workshop, ICML}, vol.~4, no.~34, 2016, p.~25.

\bibitem{Kamthe2018}
S.~Kamthe and M.~Deisenroth, ``Data-efficient reinforcement learning with
  probabilistic model predictive control,'' in \emph{International Conference
  on Artificial Intelligence and Statistics}.\hskip 1em plus 0.5em minus
  0.4em\relax PMLR, 2018, pp. 1701--1710.

\bibitem{BRL-Survey}
M.~Ghavamzadeh, S.~Mannor, J.~Pineau, A.~Tamar \emph{et~al.}, ``Bayesian
  {R}einforcement {L}earning: {A} {S}urvey,'' \emph{Foundations and
  Trends{\textregistered} in Machine Learning}, vol.~8, no. 5-6, pp. 359--483,
  2015.

\bibitem{PETS2018}
K.~Chua, R.~Calandra, R.~McAllister, and S.~Levine, ``Deep {R}einforcement
  {L}earning in a {H}andful of {T}rials using {P}robabilistic {D}ynamics
  {M}odels,'' in \emph{Advances in Neural Information Processing Systems}, vol.
  2018-December, 2018, pp. 4754--4765.

\bibitem{MBPO2019}
M.~Janner, J.~Fu, M.~Zhang, and S.~Levine, ``When to trust your model:
  Model-based policy optimization,'' in \emph{Advances in Neural Information
  Processing Systems}, vol.~32, 2019.

\bibitem{murray2018model}
R.~Murray and M.~Palladino, ``A model for system uncertainty in reinforcement
  learning,'' \emph{Systems \& Control Letters}, vol. 122, pp. 24--31, 2018.

\bibitem{Bettiol-Khalil}
P.~Bettiol and N.~Khalil, ``Necessary optimality conditions for average cost
  minimization problems,'' \emph{Discrete \& Continuous Dynamical Systems-B},
  vol.~24, no.~5, p. 2093, 2019.

\bibitem{zuazua2014averaged}
E.~Zuazua, ``Averaged control,'' \emph{Automatica}, vol.~50, no.~12, pp.
  3077--3087, 2014.

\bibitem{loheac2016}
J.~Loh{\'e}ac and E.~Zuazua, ``From averaged to simultaneous controllability,''
  in \emph{Annales de la Facult{\'e} des sciences de Toulouse:
  Math{\'e}matiques}, vol.~25, no.~4, 2016, pp. 785--828.

\bibitem{pesare2020LQR}
A.~Pesare, M.~Palladino, and M.~Falcone, ``A convergent approximation of the
  linear quadratic optimal control problem for {R}einforcement {L}earning,''
  \emph{arXiv:2011.03447}, 2020.

\bibitem{Villani}
C.~Villani, \emph{Optimal transport: old and new}.\hskip 1em plus 0.5em minus
  0.4em\relax Springer Science \& Business Media, 2008, vol. 338.

\end{thebibliography}

\end{document}